\documentclass[12pt]{article}
\usepackage{amsmath}
\usepackage{amsfonts}
\usepackage{amsthm}
\usepackage{tikz}
\usepackage{epsfig}

\newcommand{\cA}{\mbox{$\cal{A}$}}
\newcommand{\cB}{\mbox{$\cal{B}$}}
\newcommand{\cC}{\mbox{$\cal{C}$}}

\newcommand{\cG}{\mbox{$\cal{G}$}}

\newcommand{\cS}{\mbox{$\cal{S}$}}

\newcommand{\cM}{\mbox{$\cal{M}$}}

\newtheorem{theorem}{Theorem}[section]

\newtheorem{lemma}[theorem]{Lemma}
\newtheorem{corollary}[theorem]{Corollary}
\newtheorem{conjecture}[theorem]{Conjecture}

\theoremstyle{definition}
\newtheorem{definition}[theorem]{Definition}

\title{On the Density of the Set of Known Hadamard Orders}
\author{Warwick de Launey and Daniel M. Gordon\thanks{The
authors are with the IDA Center for Communications Research,
4320 Westerra Court, San Diego, CA 92121 USA 
(email: \{warwick,gordon\}@ccrwest.org).}}

\begin{document}

\maketitle

\begin{abstract}
Let $S(x)$ be the number of $n \leq x$ for which a Hadamard matrix of
order $n$ exists.  Hadamard's conjecture states that $S(x)$ is about
$x/4$.  From Paley's constructions of Hadamard matrices, we have
that
\[
S(x) = \Omega\left( \frac{x}{\log x} \right).
\]
%Asymptotic results of the form ``Hadamard matrices exist for all
%orders $2^tg$, where $g$ is odd and $t>c\log x$'' yield weaker lower
%bounds of the form $O(x^{\delta})$, where $\delta\in(0,1)$.

In a recent paper, the first author suggested that counting the
products of orders of Paley matrices would result in a greater
density.  In this paper we use results of Kevin Ford to show that 
it does: 
%one obtains
\begin{equation}\label{eq:abs}
S(x) \geq \frac{x}{\log x} \exp\left((C+o(1))(\log \log \log x)^2
\right)\,, \nonumber
\end{equation}
where $C=0.8178\ldots$.  
%In particular, 
%\[
%S(x) =\Omega\left(\frac{x}{\log x}(\log\log x)^\alpha\right)\,,
%\]
%for all $\alpha$.

%The bound \eqref{eq:abs} 
This bound
is surprisingly hard to improve upon.  We
show that taking into account all the other major known construction
methods for Hadamard matrices does not shift the bound.  Our arguments
use the notion of a (multiplicative) monoid of natural numbers.  We
prove some initial results concerning these objects.  Our techniques
may be useful when assessing the status of other existence questions
in design theory.
\end{abstract}

\nocite{crc}

\section{Introduction}
In this paper we use the idea of the density of a set of natural
numbers $\mathbb{N}$ to gauge the progress made so far on the Hadamard
Conjecture.  In addition, we propose that our methodology could be
used to assess the status of other existence problems in design
theory.

We take a moment to describe some ideas concerning (infinite) subsets
of $\mathbb{N}$, their sizes and their densities.  Section 2 covers
these and related ideas in more detail.  Given a set $\cA$ of positive
integers, we may define a {\em counting function}
$A:\mathbb{R}\rightarrow\mathbb{N}$, where $A(x)=\#\{n\leq x\;|\;n\in
\cA\}$.  The rate of growth of this function is used by number
theorists to gauge the size of the set $\cA$.  For example, the
counting function $\pi(x)$ of the set of primes is approximately equal to
$x/\log x$.
% ******************************************************************
% This isn't right; you mean numbers with only prime factors == 3(4)
% , whereas the set of squarefree numbers has associated
% function which is about $Kx/\sqrt{\log x}$, where $K=0.46459\ldots$.
% So the set of squarefree numbers is ``larger'' than the set of primes
% by a factor of $K\sqrt{\log x}$.  

In this paper, 
sets will be in calligraphic font, and the counting function for a
set will be the same letter in roman font.
We will respectively
call the function $A(x)$, and the ratio function $A(x)/x$ the {\it
size} and {\it density} (functions) of the set $\cA$.  So the set of
odd natural numbers has size about $x/2$ and density about $1/2$.  
%In general we say that the set $\cA$ has positive density $c>0$ if
%$A(x)/x>c$ for all sufficiently large $x$.  The set of primes and the
%set of squarefree numbers both have zero density even though one is
%larger and more dense than the other.

Hadamard's conjecture states:

\begin{conjecture}\label{conj:had}
  For every odd number $k$ there is a Hadamard matrix of order $2^s k$
  for $s \geq 2$.
\end{conjecture}

Let $\cS$ be the set of orders for which a Hadamard matrix exists, and
let $S(x)$ be the size function of $\cS$.  Then, since there are also
Hadamard matrices of orders 1 and 2, Conjecture~\ref{conj:had} is
equivalent to:

\begin{conjecture}\label{conj:us}
For every $x \geq 2$, 
\[
S(x) = \left\lfloor \frac{x}{4} \right\rfloor + 2\,.
\]
\end{conjecture}

There are a number of existence theorems for Hadamard matrices, but we
are far from being able to prove Conjecture~\ref{conj:us}.  The
conjecture implies that the set $\cS$ of Hadamard orders has density
$1/4$.  As yet, we have not even been able to prove that $\cS$ has
positive density.  In this paper we derive lower bounds for $S(x)$
using known existence theorems.  Using Paley's constructions we
immediately get a density of $O(x/\log x)$.  Using results on the
density of values of Euler's totient function we show:

\begin{theorem}\label{thm:main}
For all $\epsilon>0$, there is an element $x_\epsilon\in\mathbb{N}$ such 
that, for all $x>x_\epsilon$,  
\begin{equation}\label{eq:main}
S(x) \geq 
\frac{x}{\log x} \exp\left((C+\epsilon) (\log \log \log x)^2 \right)
\end{equation}
for $C = 0.8178\ldots$.
\end{theorem}

In Section~\ref{sec:impact} we show that the bound \eqref{eq:main} is
the best we can obtain given the currently known major constructions
for Hadamard matrices.  This is perhaps surprising, since
\eqref{eq:main} is obtained by taking Kronecker products of Paley
Hadamard matrices only.  So one might expect that the bound could be
improved by incorporating the many other known constructions for
Hadamard matrices.

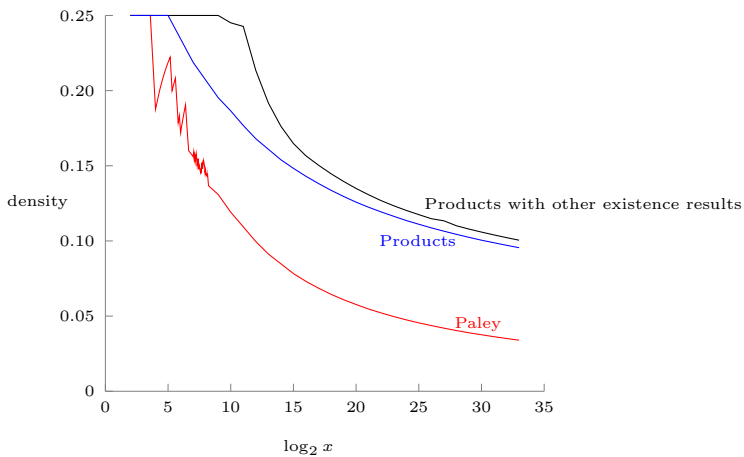
\begin{figure}[h]
  \centering
\begin{tikzpicture}[xscale=1.0,yscale=1.0]
\path[draw=red]
(0.333333,5.000000) 
--  (0.500000,5.000000) 
--  (0.597494,5.000000) 
--  (0.666667,3.750000) 
--  (0.720321,4.000000) 
--  (0.764161,4.166000) 
--  (0.801226,4.286000) 
--  (0.833333,4.376000) 
--  (0.861654,4.444000) 
--  (0.886988,4.000000) 
--  (0.909905,4.090000) 
--  (0.930827,4.166000) 
--  (0.967893,3.572000) 
--  (0.984482,3.666000) 
--  (1.000000,3.438000)
--  (1.014577,3.530000) 
--  (1.028321,3.612000) 
--  (1.041321,3.684000) 
--  (1.053655,3.750000) 
--  (1.065386,3.810000) 
--  (1.076572,3.636000) 
--  (1.097494,3.334000) 
--  (1.107309,3.200000) 
--  (1.166667,3.124000) 
--  (1.174066,3.182000) 
--  (1.181244,3.088000) 
--  (1.188214,3.142000) 
--  (1.194988,3.056000) 
--  (1.201575,3.108000) 
--  (1.207988,3.158000) 
--  (1.220321,3.000000) 
--  (1.226254,3.048000) 
--  (1.232053,3.096000) 
--  (1.243239,2.954000) 
--  (1.248642,3.000000) 
--  (1.264161,2.916000) 
--  (1.269118,2.960000) 
--  (1.273976,2.900000) 
--  (1.278737,3.040000) 
--  (1.283407,2.980000) 
--  (1.287987,3.018000) 
--  (1.292481,2.962000) 
--  (1.296893,3.000000) 
--  (1.301226,3.036000) 
--  (1.305482,3.070000) 
--  (1.317815,3.000000) 
--  (1.325699,2.904000) 
--  (1.329547,2.936000) 
--  (1.333333,2.890000) 
--  (1.340732,2.878000) 
--  (1.347911,2.868000) 
--  (1.351421,2.898000) 
--  (1.354880,2.858000) 
--  (1.358291,2.888000) 
--  (1.361654,2.848000) 
--  (1.368242,2.770000) 
--  (1.371470,2.734000) 
--  (1.500000,2.618000) 
--  (1.666667,2.382000) 
--  (1.833333,2.188000) 
--  (2.000000,1.992000) 
--  (2.166667,1.826000) 
--  (2.333333,1.696000) 
--  (2.500000,1.566000) 
--  (2.666667,1.462000) 
--  (2.833333,1.372000) 
--  (3.000000,1.290000) 
--  (3.166667,1.218000) 
--  (3.333333,1.154000) 
--  (3.500000,1.094000) 
--  (3.666667,1.042000) 
--  (3.833333,0.994000) 
--  (4.000000,0.950000) 
--  (4.166667,0.910000) 
--  (4.333333,0.874000) 
--  (4.500000,0.840000) 
--  (4.666667,0.808000) 
--  (4.833333,0.778000) 
--  (5.000000,0.752000) 
--  (5.166667,0.726000) 
--  (5.333333,0.702000) 
--  (5.500000,0.680000) ;
\draw[gray,very thin] (0.000000,0) -- (0.000000,0.1);
\draw (0.000000,0.000000) node [below] {\tiny 0};
\draw[gray,very thin] (0.833333,0) -- (0.833333,0.1);
\draw (0.833333,0.000000) node [below] {\tiny 5};
\draw[gray,very thin] (1.666667,0) -- (1.666667,0.1);
\draw (1.666667,0.000000) node [below] {\tiny 10};
\draw[gray,very thin] (2.500000,0) -- (2.500000,0.1);
\draw (2.500000,0.000000) node [below] {\tiny 15};
\draw[gray,very thin] (3.333333,0) -- (3.333333,0.1);
\draw (3.333333,0.000000) node [below] {\tiny 20};
\draw[gray,very thin] (4.166667,0) -- (4.166667,0.1);
\draw (4.166667,0.000000) node [below] {\tiny 25};
\draw[gray,very thin] (5.000000,0) -- (5.000000,0.1);
\draw (5.000000,0.000000) node [below] {\tiny 30};
\draw[gray,very thin] (5.833333,0) -- (5.833333,0.1);
\draw (5.833333,0.000000) node [below] {\tiny 35};

\draw[gray,very thin] (0,0.000000) -- (0.1,0.000000);
\draw (0.000000,0.000000) node [left] {\tiny 0};
\draw[gray,very thin] (0,1.000000) -- (0.1,1.000000);
\draw (0.000000,1.000000) node [left] {\tiny 0.05};
\draw[gray,very thin] (0,2.000000) -- (0.1,2.000000);
\draw (0.000000,2.000000) node [left] {\tiny 0.10};
\draw[gray,very thin] (0,3.000000) -- (0.1,3.000000);
\draw (0.000000,3.000000) node [left] {\tiny 0.15};
\draw[gray,very thin] (0,4.000000) -- (0.1,4.000000);
\draw (0.000000,4.000000) node [left] {\tiny 0.20};
\draw[gray,very thin] (0,5.000000) -- (0.1,5.000000);
\draw (0.000000,5.000000) node [left] {\tiny 0.25};
\draw[gray,very thin] (0.000000,0) -- (5.833333,0);
\draw[gray,very thin] (0,0.000000) -- (0,5.000000);
%all
\path[draw=black]
(0.333333,5.000000) 
--  (0.500000,5.000000) 
--  (0.666667,5.000000) 
--  (0.833333,5.000000) 
--  (1.000000,5.000000) 
--  (1.166667,5.000000) 
--  (1.333333,5.000000) 
--  (1.500000,5.000000) 
--  (1.666667,4.902000) 
--  (1.833333,4.854000) 
--  (2.000000,4.268000) 
--  (2.166667,3.836000) 
--  (2.333333,3.526000) 
--  (2.500000,3.296000) 
--  (2.666667,3.132000) 
--  (2.833333,3.008000) 
--  (3.000000,2.894000) 
--  (3.166667,2.792000) 
--  (3.333333,2.698000) 
--  (3.500000,2.614000) 
--  (3.666667,2.536000) 
--  (3.833333,2.468000) 
--  (4.000000,2.406000) 
--  (4.166667,2.350000) 
--  (4.333333,2.296000) 
--  (4.500000,2.268000) 
--  (4.666667,2.202000) 
--  (4.833333,2.158000) 
--  (5.000000,2.118000) 
--  (5.166667,2.080000) 
--  (5.333333,2.044000) 
--  (5.500000,2.010000) ;
%paley products
\path[draw=blue]
(0.333333,5.000000) 
--  (0.500000,5.000000) 
--  (0.666667,5.000000) 
--  (0.833333,5.000000) 
--  (1.000000,4.688000) 
--  (1.166667,4.376000) 
--  (1.333333,4.140000) 
--  (1.500000,3.906000) 
--  (1.666667,3.730000) 
--  (1.833333,3.536000) 
--  (2.000000,3.360000) 
--  (2.166667,3.220000) 
--  (2.333333,3.080000) 
--  (2.500000,2.966000) 
--  (2.666667,2.860000) 
--  (2.833333,2.764000) 
--  (3.000000,2.674000) 
--  (3.166667,2.592000) 
--  (3.333333,2.516000) 
--  (3.500000,2.448000) 
--  (3.666667,2.386000) 
--  (3.833333,2.328000) 
--  (4.000000,2.272000) 
--  (4.166667,2.222000) 
--  (4.333333,2.174000) 
--  (4.500000,2.130000) 
--  (4.666667,2.088000) 
--  (4.833333,2.048000) 
--  (5.000000,2.010000) 
--  (5.166667,1.976000) 
--  (5.333333,1.942000) 
--  (5.500000,1.910000) ;

\draw[black] (4.1,2.5) node [right] {\tiny Products with other existence results};

\draw[blue] (3.5,2) node [right] {\tiny Products};

%\draw[red] (2.3,4.6) node [right] {\tiny V(x)};

%\draw[black] (4.8,2.3) node [right] {\tiny $c\exp(C(\log\log\log x)^2)/\log x$};

\draw[red] (4.5,0.9) node [right] {\tiny Paley};

\draw[black] (2.7,-0.5) node [below] {\tiny $\log_2 x$};
\draw[black] (-0.9,2.75) node [below] {\tiny density};

\end{tikzpicture}
  \caption{Density of Hadamard orders from different constructions}
  \label{fig:density}
\end{figure}

Figure 1 shows plots of three lower bounds for $S(x)$.  These are
obtained by taking into account the orders of various classes of known
Hadamard matrices.  
The weakest bound is obtained using Paley orders $2^\alpha (p+1)$,
where
$\alpha \geq 1$ if $p \equiv 1 \pmod 4$, and $\alpha \geq 0$
otherwise.
The second bound takes products of these orders, and the best bound
adds 
products of the known Hadamard matrices of order up to 10000,
and the constructions described in Section~\ref{sec:impact}.
We show in that section 
that these two bounds are in fact asymptotically equal.  Indeed the
impact of the table of known orders less than 10000 seems to fade
quite rapidly.  
% The upper curve is $V(x)$, the number of values of the
% Euler totient function up to $x$, which is asymptotically equal to
% the other curves.

The Paley bound is weaker than the others, but is still stronger than
the bounds given by the asymptotic existence results proved by Seberry
and Craigen and Kharaghani.  Interestingly, Figure 1 shows that the
Hadamard Conjecture is decided in the affirmative for about one half
of the orders $n\equiv0\pmod{4}$ of size about one billion.  So at least
for ``small'' orders we are doing quite well.

We think that the notion of density has a wider applicability in the
context of design theory.  Typically, design theorists gauge the
progress on a problem by creating tables of known orders and undecided
cases.  For example, we now know that Hadamard's Conjecture holds for
nearly all orders less than $10\,000$.  However, many of the known
constructions arise from algebraic or computer constructions which may
fail to cover all cases as the upper bound on the orders to be covered
is increased.  So success for small orders may be misleading.

The existence question for Williamson matrices is a good example of
this phenomenom.  Williamson matrices of order $t$ can be used to
construct a Williamson-type Hadamard matrix of order $4t$.  In
\cite{gb} the authors obtain by computer search Williamson matrices of
order $23$, and thereby construct a Williamson-type Hadamard matrix of
order $92$.  Flushed with this success, they then suggest that
Williamson type Hadamard matrices exist for every order divisible by
four.  Indeed, subsequent computer searches confirmed that Williamson
matrices exist for all odd orders up to and including $33$.  However,
in \cite{dokovic} it was shown that no Williamson matrices exist for
order $35$, and since then additional computer searches
\cite{williamson} showed nonexistence for several more orders, so that
now the question of how common Williamson matrices are for larger
orders is quite unclear.

Therefore, there is a need for some other more global measure of the
status of a design-theoretic existence question.  Since such existence
questions usually involve two infinite sets: one consisting of the
decided orders and another consisting of the undecided orders, we
think that the discrepancy between the sizes of the set of undecided
orders and the set of decided orders provides a precise mathematical
measure of the progress made on such existence questions.

The rest of this paper is divided into three parts.  Section 2 derives
a series of lower bounds for $S(x)$.  These bounds are all implied by
Paley's construction for Hadamard matrices.  Section 2 also contains a
proof of Theorem~\ref{thm:main}.  Section 3 contains a proof that
taking into account the other major constructions for Hadamard
matrices does not lead to a larger lower bound for $S(x)$.  The proof
uses the idea of a monoid of natural numbers: i.e., a set of natural
numbers containing $1$ which is closed under multiplication.  The
final part of the paper is a technical appendix which proves two
results concerning monoids which are needed in Section 3.  The first
part of the appendix contains elementary proofs of the monoid theorems,
and the second part of the appendix gives proofs using results about
generating functions.

\section{Lower Bounds for $S(x)$ Using Paley Hadamard Matrices}
In this section, we use Paley's family of Hadamard matrices to obtain
three increasingly stronger lower bounds for $S(x)$.

\subsection{A Simple Lower Bound}
\begin{theorem}[Paley]\label{PaleyThm} 
For any prime $q$, there is a Hadamard matrix of order $n$, where
\[
n=\left\{
\begin{array}{rl}
    q+1, & \mbox{if\; $q \equiv 3 \pmod 4$,} \\
    2(q+1), & \mbox{if\; $q \equiv 1 \pmod 4$.} 
  \end{array}
\right.
\]
\end{theorem}

Dirichlet's Theorem on primes in arithmetic progressions implies the
following corollary:

\begin{corollary}\label{cor:paley}
\[
S(x) \geq \left(\frac{3}{4}+o(1)\right) \frac{x}{\log x}.
\]
\end{corollary}

\begin{proof}
One has $(1/2+o(1)) x /\log x$ orders from primes $\equiv 3 \bmod 4$ up
to $x$, and $(1/4+o(1)) x /\log x$ from primes $\equiv 1 \bmod 4$ up
to $x/2$.  An order $m$ is in both sets if $p=m-1$ and $q=m/2-1$ are
both prime.  Since $2q+1 = p$, these are Sophie Germain primes.
Brun's sieve may be used to show that the number of Sophie Germain
primes up to x is $O(x / (\log x)^2)$ (see \cite{ribenboim}), so this
overlap does not affect the density.
\end{proof}

\subsection{An Improvement}
A Hadamard matrix of order $n$ can be used to construct one of order
$2n$, so we have ones of order $2^t (q+1)$ for $t \geq 1$ for all primes
$q$.  This improves the bound in Corollary~\ref{cor:paley}:

\begin{corollary}
\[
S(x) \geq \left(\frac{3}{2}+o(1)\right) \frac{x}{\log x}.
\]
\end{corollary}
\begin{proof}
We use the following slightly stronger version of the Prime Number Theorem:
\[
\pi(x) = \frac{x}{\log x} + \frac{x}{\log^2 x} +
O\left(\frac{x}{\log^3 x}\right)\,.
\]
As before, we have $(1/2+o(1)) x /\log x$ orders from the primes
$\equiv 3 \bmod 4$ up to $x$.

Now consider the set of orders $2(p+1)$ for all $p<x/2$.
The number of these orders is
\begin{eqnarray*}
\pi(x/2) & = & \frac{x/2}{\log (x/2)} + \frac{x/2}{\log^2 (x/2)} +
O\left(\frac{x}{\log^3 x} \right) \\
&=& \frac{x}{2\log x} + \left(\frac{1 + \log 2}{2}\right) \frac{x}{\log^2 x} +
O\left(\frac{x}{\log^3 x}\right).
\end{eqnarray*}
Similarly, from all primes $p<x/2^k$ for $k< \log x$
we get 
\begin{eqnarray*}
\pi(x/2^k) & = & \frac{x/2^k}{\log (x/2^k)} + \frac{x/2^k}{\log^2 (x/2^k)} +
O\left(\frac{x}{\log^3 x} \right) \\
&=&  \frac{x}{2^k\log x} + \frac{1 + k\log 2}{2^k} \frac{x}{\log^2 x} +
O\left(\frac{x}{\log^3 x}\right).
\end{eqnarray*}
orders of the form $2^k(p+1)$.  
Summing these terms, the coefficient of $x/\log x$ converges to $3/2$,
and the coefficient of $x/\log^2 x$ also converges.

The final step is to ensure that the intersection of the sets is small:
the number of orders $m$ with $p=m/2^r-1$ and $q=m/2^s-1$ for primes
$p$ and $q$ is $o(x/\log x)$.
As for Corollary~\ref{cor:paley}, the number of such orders for any
individual $r$ and $s$ is $O(x/\log^2 x)$ using Brun's sieve.
Furthermore, we need only consider $r, s < 2 \log \log x$, since
the number of primes $p$ up to $m/2^{2 \log \log x}$ is $O(x / \log^3 x)$,
and so the number of orders $m=2^r (p + 1)$ is $O(x/\log^2 x)$.
Combining these results, the number of orders in more than one set is
$o(x/\log x)$.
\end{proof}

\subsection{Proof of Theorem~\ref{thm:main}: Further Improvements Via Products of Paley Matrices}

Given Hadamard matrices of orders $a$ and $b$, it is easy to construct
a Hadamard matrix of order $ab$, but \cite{agayan} and \cite{csz} show
that we can do better:

\begin{theorem}\label{thm:prod1}
If Hadamard matrices of order $4a$ and $4b$ exist, then there
is a Hadamard matrix of order $8ab$.
\end{theorem}

\begin{theorem}\label{thm:prod2}
If Hadamard matrices of order $4a$, $4b$, $4c$ and $4d$, exist, then there
is a Hadamard matrix of order $16abcd$.
\end{theorem}

We want to show that applying these theorems to Paley Hadamard
matrices will give us a greater density.  An improvement follows
immediately from a result of Erd\H{o}s.  He showed that the number of
different values of $m=(p+1)(q+1)$ up to $x$, for $p$ and $q$ prime,
is $(1+o(1))\frac{ x (\log \log x)}{\log x}$.  Thus
\[
S(x)\geq 
(1+o(1))\frac{ x }{\log x}(\log \log x)\,.
\]
Thus we have an immediate improvement by taking into account
Theorems~\ref{thm:prod1} and~\ref{thm:prod2}.

A further improvement follows from theory that has been developed to
analyze the distribution of values of the Euler totient function.  The
new bound (which is somewhat complicated) will imply that, for any
$\alpha>0$,
\[
S(x)\geq 
(1+o(1))\frac{ x }{\log x}(\log \log x)^\alpha\,.
\]

Recall that the Euler totient function $\varphi(n)$ is the number of
positive integers less than $n$ which are relatively prime to $n$.
This is a multiplicative function with value at prime powers:
\[ 
\varphi(p^a) = p^{a-1} (p-1)\,.
\]
Let $V(x)$ be the number of distinct values of Euler's
$\varphi$-function less than $x$.  The study of the growth of $V(x)$
has a long history.  In 1929 Pillai \cite{pillai} showed
\[
V(x) \ll \frac{x}{\log^{\log 2/e} x}\,.
\]
In 1935 Erd\H{o}s \cite{erdos} improved this to 
\[
V(x) \ll \frac{x}{\log^{1+o(1)} x}\,.
\]
The $o(1)$ was subsequently made more precise by Erd\H{o}s and Hall,
Pomerance, Maier and Pomerance, and finally Ford \cite{ford}, who showed
  \begin{eqnarray}\label{eq:vden}
V(x) &=& \frac{x}{\log x} \exp \left( C (\log \log \log x - \log \log
  \log \log x)^2 \right. \\
 &&\ \ \ \ \ \ \  \left.+ D \log \log \log x  
- (D+1/2-2C) \log \log \log \log x + O(1) \right),\nonumber
\end{eqnarray}
where $C= 0.8178\ldots$ and  $D=2.1769\ldots$.

Ford proved that this bound applies to {\em any} multiplicative
function $f$ satisfying two conditions:
\begin{eqnarray}\label{condition1}
&&  \{f(p)-p: p \ {\rm prime} \} 
\ {\rm is \ a\  finite  \ set\  not\  containing}\  0 \\
&&\sum_{h\geq16\ squareful}
\frac{\epsilon(h)}{f(h)} \ll 1, \ \epsilon(h) = \exp(\log \log h (
\log \log \log h)^{20}).\label{condition2}
\end{eqnarray}
Note that $n\in\mathbb{N}$ is squareful if, for all primes $p$,
$p\vert n$ implies $p^2\vert n$.

\begin{proof} (Proof of Theorem~\ref{thm:main})
We now use Ford's general theory to prove Theorem~\ref{thm:main}.  We
take $f(p^k)=f_2(p^k)=(p+1)^k$.  Then condition \eqref{condition1}
holds.  Moreover, $f_2(x)>\varphi(x)$; so \eqref{condition2} holds for
$f=f_2$, since it holds for $f=\varphi$.  Thus Ford's result implies,
that the set of integers up to $x$ of the form
\begin{equation}
  \label{eq:pplusone}
(p_1+1)^{\alpha_1} (p_2+1)^{\alpha_2} \cdots (p_k+1)^{\alpha_k}  
\end{equation}
has density of the same form as the righthand side of (\ref{eq:vden}).
This expression is only determined up to the ``$O(1)$'' term in the
exponent.  Nevertheless, since (by Theorems~\ref{PaleyThm} and
\ref{thm:prod1}) there are Hadamard matrices for all orders $2t$,
where $t$ has the form \eqref{eq:pplusone}, $S(x)$ is bounded below by
a function of the form on the righthand side of \eqref{eq:vden}.
Theorem~\ref{thm:main} now follows.
\end{proof}
One issue, involving powers of two, remains to be discussed.  For
each prime $p_i \equiv 1 \bmod 4$ in (\ref{eq:pplusone}), the order of
the Paley matrix is $2(p_i+1)$, not $p_i+1$.  However, this is offset
by Theorems~\ref{thm:prod1} and \ref{thm:prod2}, which show that if
$\alpha_1 + \alpha_2 + \cdots + \alpha_k = A$, we may divide
(\ref{eq:pplusone}) by a factor of two raised to the power:
$$
4\left\lfloor (A-1)/3 \right\rfloor + ((A-1) \bmod 3).
$$

Potentially this could give us an increase in our lower bound for
$S(x)$, say if we had a large number of integers in $\cS(x)$ with
$\sum \alpha_i \geq \log \log x$.  However, Ford's Theorem 10 (and its
generalization to other multiplicative functions) shows that almost
all integers in $\cS(x)$ have
\[
\sum_i \alpha_i = 2C(1+o(1)) \log \log \log x
\]
as $x \longrightarrow \infty.$
Therefore the savings from dividing out by powers of two does not
affect the main term in Theorem~\ref{thm:main}.

\section{The Impact of Other Constructions}\label{sec:impact}

In this section, we show that our best lower bound for $S(x)$ cannot
be improved by taking into account other large classes of known
Hadamard matrices.  In order to do so, we introduce the following key 
idea:

\begin{definition}
A subset $\cA$ of $\mathbb{N}$ is called a (multiplicative) monoid
if 
\begin{itemize}
\item $1\in \cA$, and 
\item $a,b\in\cA$ implies $ab\in\cA$.
\end{itemize}  
The set $\cG$ generates a monoid $\cM$ if every element in $\cM$ is a
product of elements in $\cG$.
\end{definition}
Notice that if $\cA$ and $\cB$ are monoids, then the product set
%$\cA\cB=\{ab\;:\;a\in\cA,\; b\in\cB\}$ is a monoid.
$\cA\cB=\{ab\;:\;a\in\cA,\; b\in\cB\}$ is a monoid.

Our interest in monoids stems from the observation that the set of
known Hadamard orders is closed under multiplication: i.e., the
product $n_1n_2$ of two known Hadamard orders $n_1,n_2$ is also a
known Hadamard order.  Indeed, any construction for Hadamard matrices
generates a monoid of known Hadamard orders via the Kronecker product
and the product results Theorems~\ref{thm:prod1} and~\ref{thm:prod2}.

Our overall plan in this section will be to determine the size of the
monoid generated by each major known construction, and then to
determine the size of the product of these monoids.

The following theorem allows us to determine the size of the products
of the monoids encountered in this section.  It is perhaps surprising
that 
taking finite products of monoids often does not give a
significantly larger monoid.
 
% For any subset $\cA$ of $\mathbb{N}$ and any
% $x\in\mathbb{N}$, let
% \[
% a(x)=|\cA\cap(x/2,x]|\qquad\mbox{and}\qquad\bar{a}(x)=|\cA\cap[x/2,x]|\,.
% \]

\begin{theorem}\label{corA}
Suppose that $\cA$, $\cB$, and $\cC = \cA\cB$ are monoids such that
$A(x)=O(x^{\alpha})$ and $B(x)=\Omega(x^\beta)$, where
$0<\alpha<\beta<1$.  Then $C(x)=O(B(x))$.
\end{theorem}

So up to a constant factor, the product monoid has the same size as
the larger of the two monoids.  

We will also need a result which bounds the size of a monoid in terms
of the size of its generating sets.  The next theorem shows that if a
monoid has a fairly small generating set, then the monoid itself is
not much larger.

\begin{theorem}\label{LemmaB}
Let $\cG$ be a subset of $\mathbb{N}$ such
that $G(x)=O(x^{\alpha})$, for some
$\alpha\in(0,1)$.  
Let $\cM$ be the monoid generated by
$\cG$.  Then $M(x)=O(x^{\alpha+\epsilon}$) for all $\epsilon>0$.
\end{theorem}

Notice that a monoid has a unique minimal generating set: namely, the
set of elements in the monoid which are not the product of strictly
smaller elements of the monoid.  The theorem of course applies to any
generating set.  See the Appendix for proofs of Theorems~\ref{corA}
and~\ref{LemmaB}.

The following families are given in the survey article \cite{crc}:

\begin{enumerate}
\item Hadamard matrices exist for every order $\leq 662$.  Tables of
  known orders $2^t g$ are given for odd $g<9999$.
\item A Hadamard matrix of order $2^t g$ for odd $g$ exists for 
$$t \geq 6 \left\lfloor \frac{\log_2 \frac{g-1}{2}}{16} \right\rfloor + 2.$$
\item For $g$ odd with $k$ nonzero digits in its binary expansion,
  there is a Hadamard matrix of order $2^t g$ when
  \begin{enumerate}
  \item $g \equiv 1 \pmod 4$ and $t \geq 2k$,
  \item $g \equiv 3 \pmod 4$ and $t \geq 2k-1$.
  \end{enumerate}
\item 
For $q$ a prime power, $q \not \equiv 7 \pmod 8$ a Hadamard matrix of
order $4q^2$ exists.
\item For $q$ odd, a Hadamard matrix of order $4q^4$ exists.
\item If $n-1$ and $n+1$ are both odd prime powers, then there exists
  a Hadamard matrix of order $n^2$.
\end{enumerate}

We also note the large class of cocyclic\footnote{Cocyclic Hadamard
matrices correspond to certain relative difference sets.} Hadamard
matrices:
\begin{enumerate}
\item[7] Let $p_1,p_2,\dots,p_r\equiv1\pmod{4}$ and let
$q_1,q_2,\dots,q_s\equiv3\pmod{4}$ be prime powers.  Then, for all
$\alpha_1,\alpha_2,\dots,\alpha_r,\beta_1,\beta_2,\dots,\beta_s\geq0$,
there is a cocylic Hadamard matrix of order
\[
\prod_{i=1}^r  2p_i^{\alpha_i}(p_i+1)
\prod_{i=1}^s  q_i^{\beta_i}(q_i+1)\,.
\]
\end{enumerate}

We first observe that the orders in the last class form a monoid
$\cM_7$ whose size has the same form as $V(x)$.  To see this, we
define $f_3(p^k)=p^{k-1}(p+1)$, and then apply Ford's theorem.  Notice
that $f_3(x)>\varphi(x)$; so condition \eqref{condition2} holds for
$f_3$ since it holds for the totient function $\varphi$.  

Notice also that, if $\cM'_7$ includes all the orders obtained by
applying Theorems~\ref{thm:prod1} and~\ref{thm:prod2} to the orders
listed under item 7, then $\cM'_7$ contains all the orders identified
in the previous section.  Moreover, the argument at the end of Section
2 implies that $\cM_7$ and $\cM'_7$ have about the same size.

We now show that the Hadamard orders given by constructions 1--6 in
combination generate a monoid whose size is quite small.

\begin{theorem} \label{thm:others}
The monoid $\cM$ generated by all the orders given by constructions
1--6 has size $O(x^{8/11+\epsilon})$, where $\epsilon>0$ may be taken
as close to zero as one pleases.
\end{theorem}
\begin{proof}
We consider the constructions 1--6 listed above in order:
\begin{enumerate}
\item The first construction generates a monoid $\cM_1$ which has a
finite number of generators.  So $\cM_1\cap[1,x]$ has size $O((\log
x)^a)$, where $a$ is the number of generators.
\item Let $\cM_2$ be the monoid generated by the set $\cG_2$ of orders
given by the second construction.  Then $\cM_2$ is not much smaller
than the set
\[
\{2^tg\;|\; \mbox{where $g$ is odd and $t\ge
\epsilon\log_2 g$}\}\,,
\]
where $\epsilon=3/8$.  The number of elements in this set is about
equal to
\[
\sum_{g^{1+\epsilon}\ odd\ \leq x}\log_2(x/g^{1+\epsilon})
=O(x^{\tfrac{1}{1+\epsilon}}\log x)\,.
\]
\item To assess the size of the monoid $\cM_3$ given by this
construction, we apply Theorem~\ref{LemmaB}.  

Let $\cG_3$ be the set of orders $2^tg$ satisfying parts (a) and (b)
of item 3 above.  Then $\cG_3$ generates $\cM_3$.  We now estimate the
size of $\cG_3\cap[1,x]$.  Put $n=\lceil\log_2 x\rceil$, and suppose
$2^tg\in\cG_3\cap[1,x]$.  If $g \equiv 1 \pmod 4$ has has exactly $k$
digits equal to $1$, then the bottom $2k$ digits of the binary
expansion of $2^t g$ must be zero, and the remaining $n-2k$ digits
must contain exactly $k$ $1$s.  This gives at most ${n-2k\choose k}$
possibilities.  If $g\equiv 3\pmod{4}$ has exactly $k$ digits equal to
$1$, then the bottom $2k-1$ digits of the binary expansion of $2^t g$
must be zero, and the remaining $n-2k+1$ digits must contain exactly
$k$ $1$s.  This gives at most ${n-2k+1\choose k}$ possibilities.  So
\[
G_3(x)\leq\sum_{k}
{{n-2k+1}\choose k} + {{n-2k}\choose{k}}.    
\]
There are at most $n$ summands, and the largest of these occurs when
$k\approx n/4$.  So $G_3(x)=O(x^{\tfrac{1}{2}+\epsilon})$ for any
$\epsilon>0$.  Theorem~\ref{LemmaB} now implies that
$M_3(x)=O(x^{\tfrac{1}{2}+\epsilon})$ for any $\epsilon>0$.
\item
Constructions 4, 5 and 6 all give orders lying in the monoid $\cM_4$ of
square orders.  So the monoid generated by these orders has size
$O(x^{1/2})$.
\end{enumerate}

The monoids $\cM_1,\cM_2,\cM_3$ and $\cM_4$ all have size
$O(x^\delta)$ where $\delta>8/11$ may be taken as close to $8/11$ as
one pleases.  So Theorem~\ref{LemmaB} implies the result.
\end{proof}

From 
Theorem~\ref{thm:others} 
and 
Theorem~\ref{corA} 
we see that constructions 1--7 do not increase the
asymptotic bound for $S(x)$:

\begin{theorem}
The monoid $\cM_0$
generated by constructions 1--7 has size 
$$
\frac{x}{\log x} \exp\left((C+o(1)) (\log \log \log x)^2 \right)
$$
for $C = 0.8178\ldots$.
\end{theorem}

\section*{Acknowledgements}

The authors would like to thank Carl Pomerance for directing them to
the literature on the distribution of Euler's function.

\bibliography{warwick}
\bibliographystyle{plain}

\appendix

\section{Monoids and Sets of Natural Numbers}\label{sec:monoid}

In this appendix, we 
prove two theorems showing that taking products of sets does not
greatly increase asymptotic density.  We give two sets of proofs; one
elementary and self-contained, and the other shorter but depending on
results on generating functions.

% As noted in the introduction, any set $\cA$ of natural numbers
% has an counting function $A:\mathbb{N}\rightarrow\mathbb{N}$,
% defined so that for all $x\in\mathbb{N}$, $A(x)=\#\{\;n\leq
% x\;:\;n\in\cA\;\}$.  

%In this paper, we will suppose that our sets
%always contain $1$.  This ensures that the size function is always
%strictly positive.  

% It is sometimes convenient to consider the ratio
% $A(x)/x$.  We call this the density (function) of $\cA$.  

In this paper, we use some standard notation to discuss the growth of
the counting function of a set:
Let $f:\mathbb{N}\rightarrow\mathbb{R}$ be a function.  Then
\begin{itemize}
\item ``$A(x)=O(f(x))$'' means that there is a constant $C>0$ and
$x_0\in\mathbb{N}$ such that $A(x)<C f(x)$ for all $x\geq x_0$,
\item ``$A(x)=\Omega(f(x))$'' means that there is a constant $C>0$ and
$x_0\in\mathbb{N}$ such that $A(x)>C f(x)$ for all $x\geq x_0$.
\item ``$A(x)=\Theta(f(x))$'' means that there are constants $c_1>c_2>0$
and $x_0\in\mathbb{N}$ such that $c_1f(x)\geq A(x)\geq c_2f(x)$ for all 
$x\geq x_0$.
\item ``$A(x)=o(f(x))$'' means that for any constant $C>0$ there is some
$x_0\in\mathbb{N}$ such that $A(x) < C f(x)$ for all 
$x\geq x_0$.
\end{itemize}
% Notice that, if $f(n)>0$ for all $n\in\mathbb{N}$, then we may choose
% the constants so that $x_0=1$.  Note also that, if $A(x)=O(f(x))$ and
% $A(x)=\Omega(f(x))$, then $A(x)=\Theta(f(x))$.  So, if $\pi(x)$ is the
% number of primes less than $x$, then, by the Prime Number Theorem,
% $\pi(x)=O(x/\log x)$, $\pi(x)=\Omega(x/\log x)$ and therefore
% $\pi(x)=\Theta(x/\log x)$.  Notice also that $\Theta$ is transitive
% and symmetric whereas $O$ and $\Omega$ are transitive but not
% symmetric.\footnote{We mention one final point concerning this
% notation.  Often the function $f(n)$ will not be defined for all
% $n\in\mathbb{N}$.  For example $x/\log x$ is not defined for $x=1$.
% Later on we will use quantities such as $\log\log x$, and even
% $\log\log\log\log x$.  This latter term is not defined until $x\geq
% e^{e^{e}}\approx 3 814 280$.  In these cases, if we needed $f(n)$ to
% be defined for all of $\mathbb{N}$, then we could replace these
% quantities with the respective quantities $x/\log(1+x)$,
% $\log\log(x+e)$, and $\log\log\log\log(x+e^{e^{e}})$.  These all have
% the same asymptotic behavior as the quantities which they replace.}

\subsection{Elementary Proofs}

For any subset $\cA$ of $\mathbb{N}$ and any
$x\in\mathbb{N}$, let
\[
a(x)=|\cA\cap(x/2,x]|\qquad\mbox{and}\qquad\bar{a}(x)=|\cA\cap[x/2,x]|\,.
\]

\begin{lemma}\label{ProductDensityBoundLemma}\label{LemmaA}
Let $\cA,\cB$ and $\cC=\cA\cB$ be subsets of $\mathbb{N}$ which are monoids.
Then, for all $x\in\mathbb{N}$, 
\begin{equation}\label{ProductDensityBound}
\frac{c(x)}{b(x)}
\leq  \sum_{k=1}^{\lceil\log_2x\rceil}
    \frac{(b(x/2^{k-1})+b(x/2^k))\bar{a}(2^k)}{b(x)}\,.
\end{equation}
Moreover, if the righthand side is bounded by a constant $c_1$, say,
for all $x\in\mathbb{N}$, then $C(x)=\Theta(B(x))$.
\end{lemma}

\begin{proof}
Since every element of $\cC\cap(x/2,x]$ can be written in the form
$ab$, where, for some $k\in\{1,2,\dots,\lceil\log_2x\rceil\}$, $a\in
\cA\cap[2^{k-1},2^k]$ and $b\in \cB\cap(x/2^{k+1},x/2^{k-1}]$, we have
\[
c(x)
\leq  \sum_{k=1}^{\lceil\log_2x\rceil}
    (b(x/2^{k-1})+b(x/2^k))\bar{a}(2^k)\,.
\]
Dividing through by $b(x)$ then gives \eqref{ProductDensityBound}.
We now prove the second part of the lemma.  By hypothesis, we have
\begin{align*}
c(x)
\leq b(x)\left\{  \sum_{k=1}^{\lceil\log_2x\rceil}
    \frac{(b(x/2^{k-1})+b(x/2^k))\bar{a}(2^k)}{b(x)} \right \}
\leq c_1 b(x)\,.
\end{align*}
Now we have the following partition 
\[
\cC\cap[1,x]=\bigcup_{k=1}^{\lceil\log_2 x\rceil}\cC\cap(x/2^k,x/2^{k-1}]
\]
for $\cC\cap[1,x]$ and a similar partition for $\cB\cap[1,x]$.
So
\[
C(x)= \sum_{k=1}^{\lceil\log_2x\rceil} c(x/2^{k-1})
\leq c_1\sum_{k=1}^{\lceil\log_2x\rceil} b(x/2^{k-1})
=c_1B(x)\,.
\]
Since $\cB\subset\cC$, we then have $B(x)\leq C(x) \leq c_1B(x)$.
This completes the proof of the second part of the lemma.
\end{proof}
% As Theorem~\ref{corA} shows, monoids of size $\Theta(x^\alpha)$
% ($\alpha\in(0,1)$) have an interesting property: product monoids of
% these monoids are not much bigger than the biggest constituent monoid.

{\em Proof of Theorem~\ref{corA}:}
For some constants $c_1,c_2>0$,
\[
b(x)
= B(x)-B(\lfloor{x/2}\rfloor)
\geq c_1 x^{\beta}-c_2 \left(\frac{x}{2}\right)^{\beta}
= \left(\frac{x}{2}\right)^{\beta}(2^{\beta}c_1-c_2)
\]
Now
\begin{align*}
\sum_{k=1}^{\lceil\log_2x\rceil}
    \frac{(b(x/2^{k-1})+b(x/2^k))\bar{a}(2^k)}{b(x)}
&\leq \sum_{k=1}^{\lceil\log_2x\rceil}\frac{B(x/2^{k-1})A(2^k)}{b(x)}\\
&\leq c_3 \sum_{k=1}^{\lceil\log_2x\rceil}
        \left(\frac{x}{2^{k-1}}\right)^{\beta}2^{k\alpha}
        \left(\frac{2}{x}\right)^{\beta}\\
&\leq c_4 \sum_{k=1}^{\lceil\log_2x\rceil} 
     2^{k(\alpha-\beta)}\,,
\end{align*}
which is bounded since $\alpha<\beta$.  So
Lemma~\ref{ProductDensityBoundLemma} implies that $C(x)=\Theta(B(x))$.

We now prove Theorem~\ref{LemmaB}, that 
the size of a monoid is at most slightly bigger than its generating set:

{\em Proof of Theorem~\ref{LemmaB}:}
Fix $\epsilon>0$.  We prove $M(x)=O(x^{\alpha+\epsilon})$.  Put
$\alpha_0=\alpha+\epsilon/2$.  Let $x_0$ be such that $G(x)\leq
\tfrac{1}{2}x^{\alpha_{0}}$ for all $x\geq x_{0}$.  Let
$\cG_{0}=\cG\cap [1,x_0)$, and let $\cG_{1}=\cG \cap [x_0,\infty)$.
Let $\cM_{0}$ be the monoid generated by $\cG_{0}$, and let $\cM_{1}$
be the monoid generated by $\cG_{1}$.  Then the following statements
hold:
\begin{enumerate}
\item[(A)]
$G_1(x) \leq \tfrac{1}{2}(x^{\alpha_0})$,
\item[(B)]
$M_{0}(x)=O((\log x)^{|\cG_0|})$,
\item[(C)]
$\cM = \cM_{0}\cM_{1}$,
\item[(D)]
$M(x)\leq M_0(x)M_1(x)=O((\log x)^{|\cG_0|}M_1(x))$.
\end{enumerate}
So, noting item (D), in order to prove that $M(x)=O(x^{\alpha+\epsilon})$, it is
sufficient to prove that $M_1(x)=O(x^{\alpha_1})$, for all
$\alpha_1\in(\alpha_0,\alpha+\epsilon)$.

Fix $\alpha_1\in(\alpha_0,\alpha+\epsilon)$.  We prove
$M_1(x)=O(x^{\alpha_1})$.  Let $n = \lceil\log_{2}x\rceil$.  Any
element $y$ of $\cM_{1} \cap [1,x]$ corresponds to a partition of $n$
as follows: Suppose $y = y_{1}y_{2} \dots y_{r}$ where $y_{1} \leq
y_{2} \leq \dots\leq y_{r}$ are elements of $\cG_1$.  Put $a_{i}=
\lfloor\log_{2}y_{i}\rfloor$.  Then $a_1+a_2+\dots a_r=m\leq n$, and
$0\leq a_1\leq a_2\leq\dots\leq a_r$.  Thus replacing $a_r$ with
$a_r'=a_r+n-m$, we see that any product $y=y_1y_2\dots
y_r\in\cM_1\cap[1,x]$ of $r$ elements $y_i$ of $\cG_1$ maps to a
partition of $n$ into at most $r$ pieces.  The number of such $y$
sequences $y_1,y_2,\dots,y_r$ with $\lfloor\log_2 y_i\rfloor = a_i$ is at
most
\[
G_1(2^{a_1+1}) G_1(2^{a_2+1}) \dots  G_1(2^{a_r+1})
\leq
2^{a_1\alpha_0}2^{a_2\alpha_0}\dots2^{a_r\alpha_0}
\leq 2^{n\alpha_0}\,.
\]
Now Hardy and Ramanujan showed that the number $p(n)$ of partitions of
$n$ is asymptotic to
\[
\exp(\pi\sqrt{2n/3})/4n\sqrt{3}=O(x^{\delta})\,,
\]
for all $\delta>0$.  So, choosing $\delta=\alpha_1-\alpha_0$,
we have
\[
M_1(x)\leq p(n) 2^{\lceil\log_2 x\rceil\alpha_0}
= O(x^{\alpha_0+\alpha_0})=O(x^{\alpha_1})\,.
\]

\subsection{Proofs using Generating Functions}

We will use generating functions
to show that these constructions do not increase the density of known
Hadamard orders.  
Since we are interested in the properties of products of
sets $\cC = \cA \cB$, functions of the form
$$
%a(z) =
 \sum_{n \in \cA} z^{\log_2 n} 
$$
are useful, since multiplying elements corresponds to adding the
powers in terms of the series.  
In the context of smooth numbers, 
Bernstein \cite{bernstein} estimated such functions by looking at
$$
a(z) = \sum_{k \geq 0}  a_k z^k := \sum_{n \in \cA} z^{\lfloor \log_2 n\rfloor} .
$$
These series have many fewer terms, and so are easier to analyze.
Note that $a_k = A(2^k) - A(2^{k-1})$ is the number of $k$-bit elements of
$\cA$.  We will prove results about $a_k$, i.e. $A(x)$ for $x$ a power
of two, but since $A(x)$ is monotone increasing, and all the
coefficients of the generating function are nonnegative, this will suffice.

\begin{lemma}\label{lem:abc}
  Let $\cC$ be the set of products of
  elements of sets $\cA$ and $\cB$ with series $a(z)$ and $b(z)$.  
Then 
$$
c(z) \leq a(z) \frac{b(z)}{1-z}.
$$
\end{lemma}

\begin{proof}
The coefficient of $z^n$ in 
$a(z) \frac{b(z)}{1-z} = a(z) b(z) (1 + z + z^2 \cdots)$
is
$$
\sum_{k=0}^n a_k B(2^{n-k}).
$$
Any $n$-bit element of $\cC$ can be written as a product of a $k$-bit
element of $\cA$ and an element of $\cB$ of at most $n-k$ bits.
\end{proof}

% A {\em series over $\Qnum$}  is a function $f: \Rnum \longrightarrow
% \Qnum$ such that $\{n \leq h : f(n) \neq 0\}$ is finite for every $h
% \in \Rnum$.  
% We say that $f \geq g$ if $f(n) \geq g(n)$ for all $n$.
% Bernstein used series to get precise estimates for the numbers of
% smooth integers up to $x$; we will use them to get asymptotic results.

% We will look at sums such as
% $$
% s(z) := \sum_{n \in S(x)} z^{\lfloor \log_2 n\rfloor} 
% = \sum_{k \geq 0} s_k z^k
% $$
% By Theorem~\ref{thm:main}, 
% \begin{equation}
%   \label{eq:sk}
% s_k = S(2^k) - S(2^{k-1}) = \frac{(1+o(1)) 2^{k-1}}{k\log 2}.  
% \end{equation}

% The {\em distribution of terms of $f$} $D(f)$ is the function
% $h \longrightarrow \sum_{n \leq h} f_n$.  From Theorem 2.4 of
% \cite{bernstein} we have:

% \begin{lemma}
% Let $f$ and $g$ be series over $\Qnum$ with 
% $f,g \geq 0$, and $D(f) \geq D(g)$ for all $i$.
% Then
% $$D(f) \geq D(g).$$
% \end{lemma}

We may use the analytic properties of series like this to bound the
size of the corresponding counting function.
Flajolet and Sedgewick \cite{fs} give a wealth of such results.  
Their Theorem IV.7 relates the growth rate of power series
coefficients to singularities of the corresponding function:

\begin{theorem}\label{thm:bender}
If $f(z)= \sum f_n z^n$ has positive coefficients and is analytic at
$0$ and
$$
R = \sup \{ r \geq 0 | f {\rm \ is \ analytic \ at \ all \ points \
  of\ } 0 \leq z < r\}
$$
then $\lim \sup |f_n|^{1/n} = (1/R)$.
\end{theorem}

%From (\ref{eq:sk}) 

From Corollary~\ref{cor:paley} we have $s_k = \Theta(2^{(1-\epsilon)
k})$ for any $\epsilon>0$, so by the ratio test the radius of
convergence of $s(z)$ is $1/2$.  The coefficients of generating
functions for the other monoids have smaller growth, and so a larger
radius of convergence.  Theorem VI.12 of \cite{fs} shows that the size
of the product set $\cA \cB = \{ ab | a \in \cA, b \in \cB \}$ of two
sets with different growth rates is a constant times the size of the
larger set, proving Theorem~\ref{corA}:

\begin{theorem}\label{corA2b}
  Suppose $a(z) = \sum a_n z^n$ and $b(z) = \sum b_n z^n$ are power
series with radii of convergence $\alpha > \beta \geq 0$,
respectively.  Suppose $b_{n-1}/b_n$ approaches a limit $b$ as $n
\longrightarrow \infty$.  If $a(b) \neq 0$, then $c_n \sim a(b) b_n$,
where $\sum c_n z^n = a(z) b(z)$.
\end{theorem}

Finally consider the monoid generated by a set $\cA$.  The generating
function for the monoid will be
\[
{\rm Exp}(a(z)) := \exp \left( a(z) + \frac{1}{2} a(z^2) +
  \frac{1}{3} a(z^3) +  \cdots \right).
\]
This function has the same radius of convergence as $a(z)$ (see
Section IV.4 of \cite{fs}), giving Theorem~\ref{LemmaB}.

% \begin{theorem}\label{LemmaB}
% Let $\alpha\in(0,1)$, and 
% let $\cG$ be a subset of $\mathbb{N}$ such
% that $G(x)=O(x^{\alpha})$.  Let $\cM$ be the monoid generated by
% $\cG$.  
% Then $M(x)=O(x^{\alpha+\epsilon}$) for all $\epsilon>0$.
% \end{theorem}

\end{document}